\documentclass{article} 

 \usepackage[usenames,dvipsnames]{pstricks}
 \usepackage{epsfig}
 \usepackage{pst-grad} 
 \usepackage{pst-plot} 
 \usepackage[space]{grffile} 
 \usepackage{etoolbox} 
 \makeatletter 
 \patchcmd\Gread@eps{\@inputcheck#1 }{\@inputcheck"#1"\relax}{}{}
 \makeatother

\usepackage[usenames,dvipsnames]{pstricks}
\usepackage{epsfig}
\usepackage{pst-grad} 
\usepackage{pst-plot} 
\usepackage[space]{grffile} 
\usepackage{etoolbox} 
\makeatletter 
\patchcmd\Gread@eps{\@inputcheck#1 }{\@inputcheck"#1"\relax}{}{}
\makeatother

 \usepackage[usenames,dvipsnames]{pstricks}
 \usepackage{epsfig}
 \usepackage{pst-grad} 
 \usepackage{pst-plot} 
 \usepackage[space]{grffile} 
 \usepackage{etoolbox} 
 \makeatletter 
 \patchcmd\Gread@eps{\@inputcheck#1 }{\@inputcheck"#1"\relax}{}{}
 \makeatother

\usepackage{amsmath}
\usepackage{amssymb}
\usepackage{amsthm}
\usepackage[latin1]{inputenc}
     
\usepackage{graphicx}

\usepackage[usenames,dvipsnames]{pstricks}
\usepackage{epsfig}
\usepackage{pst-grad} 
\usepackage{pst-plot} 

\usepackage{ifthen}
\usepackage{color}

\newcommand{\intav}[1]{\mathchoice {\mathop{\vrule width 6pt height 3 pt depth  -2.5pt
\kern -8pt \intop}\nolimits_{\kern -6pt#1}} {\mathop{\vrule width
5pt height 3  pt depth -2.6pt \kern -6pt \intop}\nolimits_{#1}}
{\mathop{\vrule width 5pt height 3 pt depth -2.6pt \kern -6pt
\intop}\nolimits_{#1}} {\mathop{\vrule width 5pt height 3 pt depth
-2.6pt \kern -6pt \intop}\nolimits_{#1}}}

\def\polhk#1{\setbox0=\hbox{#1}{\ooalign{\hidewidth\lower1.5ex\hbox{`}\hidewidth\crcr\unhbox0}}}

\renewcommand{\div}{\operatorname{div}}

\renewcommand{\div}{\operatorname{div}}

\newtheorem{teo}{Theorem}

\newtheorem{Definition}{Definition}

\newtheorem{Proposition}{Proposition}
\newtheorem{Remark}{Remark}
\newtheorem{Assumption}{A}

\begin{document}

\title{Improved regularity for the porous medium equation along zero level-sets}
\author{Edgard A. Pimentel and Makson S. Santos}

\date{\today} 

\maketitle

\begin{abstract}

\noindent In the present work we establish sharp regularity estimates for the solutions of the porous medium equation, along their zero level-sets. We work under a proximity regime on the exponent governing the nonlinearity of the problem. Then, we prove that solutions are locally of class $\mathcal{C}^{1-,\frac{1}{2}-}$ along free boundary points $x_0\in \partial\left\lbrace u>0\right\rbrace$, both in time and space. Our argument consists of importing information from the heat equation, through approximation and localization methods. 

\noindent 
\medskip


\medskip 

\end{abstract}

\vspace{.1in}

\section{Introduction}\label{introduction}
In this paper we examine an inhomogeneous porous medium equation (PME) of the form 
\begin{equation}\label{eq_main2}
	u_t - \Delta(u^m) = f \;\;\;\;\;\mbox{in}\;\;\;\;\; Q_1,
\end{equation}
where $f\in L^{p,q}(Q_1)$, $m>1$ and $Q_1:=B_1\times(-1,0]$.

We prove new (sharp) regularity results for the solutions to \eqref{eq_main2}. In particular, we establish gains of regularity as solutions approach their zero level-set 
\[
	S_0(u)\,:=\,\left\lbrace(x,t)\,\in\,Q_1\;|\; u(x,t)\,=\,0\right\rbrace.
\]
More precisely, we show that weak solutions to \eqref{eq_main2} are locally of class $\mathcal{C}^{1-}(Q_1)$ along the associated interface. The underlying motivation for the present work is to import regularity from the \textit{homogeneous heat equation} back to the PME, in line with ideas introduced by \textsc{L. Caffarelli} in \cite{caffarelli89}. 

Recently, those intrinsically geometric methods have been developed into a more general analytical framework. Led by the works of \textsc{E. Teixeira}, \textsc{J.-M. Urbano} and their collaborators, the \emph{geometric tangential analysis} has been proven instrumental in unveiling further layers of information on various classes of problems; see \cite{cpimrn}, \cite{edupota}, \cite{pimpoin}, \cite{cpprime}, \cite{argledu}, \cite{pimtei}, \cite{silteix}, \cite{teixurb_anal}, \cite{univmoduli}, just to mention a few references. We also allude to the survey papers \cite{teixeirasurvey}, \cite{pimsan}, \cite{teixurbsurvey} and the references therein. The present work is inspired by those methods and advances them further into the context of the porous medium equation.

The PME appears in various settings, ranging from fluid mechanics to differential geometry. In arbitrary dimensions, it models the flow of a gas through a porous medium. In \eqref{eq_main2}, the unknown $u$ stands for the density of the gas, whereas
\[
	v(x,t)\,:=\,u^{m-1}(x,t)
\]
is the pressure of the fluid. In the context of applications to mathematical disciplines, we highlight the importance of the PME in differential geometry. In fact, when $d=2$ and $m=0$, the equation
\[
	u_t\,-\,\div\left(u^{m-1}Du\right)\,=\,0
\]
describes the Ricci flow on a surface. In addition, in case $d\geq 3$ and the exponent $m$ satisfies
\[
	m\;:=\,\frac{d\,-\,2}{d\,+\,2},
\]
this equation describes the Yamabe flow. For the relation of the PME with the Ricci flow, we refer the reader to \cite{wuricci} and \cite{totiyamabe}; for the connection of this equation with the Yamabe flow, we mention \cite{vazquezyamabe}.

Due to its rich -- though (very) simple -- nonlinear structure, \eqref{eq_main2} has attracted the attention of several authors, working on various aspects of the problem. Important developments have been produced in the literature. These include the unique solvability of \eqref{eq_main2} in the class of weak (distributional) solutions, finite speed of propagation, the study of particular (self-similar) solutions, the Aronson-B\'enilan and Harnack inequalities, among others. For a fairly complete account of the mathematical theory available for \eqref{eq_main2}, we refer the reader to \cite{vazquezmono} and \cite{totikenig}. A detailed discussion on degenerate/singular evolution equations -- including \eqref{eq_main2} -- is put forward in \cite{urbanocute}. In \cite{urbanocute2}, the author presents the method of intrinsic scaling as an approach to access information on the regularity of the solutions to degenerate and singular PDEs; the PME in particular is considered in this monograph.

Among the families of self-similar solutions to \eqref{eq_main2}, we highlight the so-called \emph{Barenblatt} solutions. These are of the form
\begin{equation}\label{eq_barenblatt}
	b(x,t)\,:=\,\frac{1}{t^\alpha}\left(M\,-\,k\frac{|x|^2}{t^\beta}\right)^{\frac{1}{m-1}}_+,
\end{equation}
where
\[
	\alpha\,:=\,\frac{d}{d(m\,-\,1)\,+\,2}\;\;\;\;\;\;\;\;\;\;\mbox{and}\;\;\;\;\;\;\;\;\;\;\beta\,:=\,\frac{1}{d(m\,-\,1)\,+\,2}.
\]
The relevance of the Barenblatt solutions is due in part to the information they imply on the PME. The most remarkable one concerns the emergency of the free boundary 
\[
	\partial\left\lbrace (x,t)\,\in\,Q_1\;|\;u(x,t)\,>\,0\right\rbrace\,\cap\,Q_1,
\]
consequential on the finite speed of propagation in the presence of compactly supported initial conditions. In addition, and of particular interest in the context of this paper, is the H\"older continuity of the solutions.

The regularity of the (weak) solutions to \eqref{eq_main2} is a central topic in the analysis of (degenerate) PDEs. For $m\neq 1$ the diffusivity coefficient $mu^{m-1}$ vanishes along $\mathcal{S}$. Therefore, the equation degenerates along the zero level set. As a consequence, standard results in elliptic regularity theory (e.g., De Giorgi-Nash-Moser) are no longer available. 

In the unidimensional setting, it is known that the pressure function is Lipschitz continuous with respect to the space variable \cite{Ar1}. As a corollary, it follows that solutions are $\alpha$-H\"older continuous, with exponent $\alpha$ given by
\[
	\alpha\,:=\,\min\,\left\lbrace 1,\,\frac{1}{m\,-\,1} \right\rbrace.
\]
The connection between the regularity in space and the regularity in time is the subject of \cite{kruzhkov}; see also \cite{gilding}. Lipschitz regularity of the pressure function with respect to time is established in \cite{dibenedetto} and \cite{ArCa}. Refer also to \cite{benilan}. In \cite{caffried}, the authors obtain improved regularity of the free boundary $\partial\left\lbrace u>0\right\rbrace$, by showing that it is continuously differentiable. 

Regularity for the PME in higher dimensions is firstly studied in \cite{caffried2}. In that work the authors prove the existence of a modulus of continuity for the solutions to \eqref{eq_main2}. Indeed, they establish the existence of $\delta_0>0$ so that
\begin{enumerate}
\item if $d\,=\,2$, $u$ has a modulus of continuity $\omega_2:\mathbb{R}\to\mathbb{R}$ given by
\[
	\omega_2(r)\,:=\,C2^{-C|\ln r|^{1/2}},
\]
for some $C>0$;
\item if $d\,>\,2$, $u$ has a modulus of continuity $\omega_d:\mathbb{R}\to\mathbb{R}$ of the form
\[
	\omega_d(r)\,:=\,C|\ln r|^{-\varepsilon},
\]
for some $C>0$ and $\varepsilon\in(0,2/d)$.
\end{enumerate}

The H\"older regularity of the solutions and the free boundary is the subject of \cite{caffried3}. In that paper, the authors explore the geometry of the interface and prove that it is H\"older continuous. In addition they refine the moduli of continuity obtained in \cite{caffried2} and prove H\"older continuity of the density $u$.

Lipschitz regularity of the free boundary is the object of \cite{cafvazwo}. In that work, the authors prove that the interface is a Lipschitz surface of codimension $1$ in $\mathbb{R}^{d+1}$, for $t>T_0$. Here, $T_0>0$ is related to the support of the initial condition. In \cite{cafwo}, the authors prove the $\mathcal{C}^{1,\alpha}$-regularity of the free boundary, under the same set of assumptions in \cite{cafvazwo}.

The smoothness of the free boundary is established in \cite{daskham}. In that paper, the authors suppose the initial condition for the pressure to satisfy natural conditions when restricted to its compact support. Such conditions are related to upper bounds on the initial pressure and its derivatives up to the second order. In this context, they prove the existence of an instant $T>0$, such that the pressure function and the interface are of class $\mathcal{C}^\infty$, for $0<t<T$. We also refer the reader to \cite{daskham2}.

In \cite{kikova}, the authors assume flatness of the solutions and examine regularity properties of the pressure function and the free boundary. To be more precise, they prove the existence of an instant $\tau>0$ such that solutions are $\mathcal{C}^\infty$-regular in the positivity set and up to the interface for times $t>\tau$. Moreover, they derive $\mathcal{C}^\infty$-regularity of the free boundary after time $\tau$.

Nonlocal variations of the PME are investigated in \cite{bofiva}. In that context, the authors examine a number of important matters regarding the solutions. This range from a priori estimates to Harnack inequalities, including a quantitative analysis of the boundary behavior. In \cite{gianazza}, the authors consider $m\geq 2$ and impose a positivity condition on the average of the solutions. Under those conditions, they prove uniform (local) bounds for the gradient of the pressure function. 

Distinct regularity classes, and their effects on the regularity of the solutions, are examined in \cite{bogelein}. In that paper, the authors establish an equivalence between weak solutions defined under the conditions
\[
	u^m\,\in\,L^2_{loc}(0,T;H^1_{loc}(B_1))
\]
and
\[
	u^\frac{m+1}{2}\,\in\,L^2_{loc}(0,T;H^1_{loc}(B_1)).
\]

In \cite{urbdamma}, the authors examine the optimal regularity of the solutions to \eqref{eq_main2} in terms of the optimal regularity for the homogeneous PME. They argue through a geometric set of techniques, by importing information from the homogeneous counterpart of \eqref{eq_main2}.
%
%

In the present paper we examine the regularity of the solutions to \eqref{eq_main2} \emph{along the degeneracy set}
\[
	S_0(u)\,:=\,\left\lbrace (x,t)\,\in\,Q_1\;|\;u(x,t)\,=\,0\right\rbrace.
\]

The set of techniques comprised by the geometrical tangential methods was proven effective in unveiling improved, sharp, regularity results for the solutions to degenerate/singular PDEs along their critical sets. The pivotal contribution in this direction is reported in \cite{gradsing}. In that paper, the author examines the solutions to an elliptic problem related to
\[
	\div(|Du|^{p-2}Du)\,=\,f,
\]
where $f\in L^q(B_1)$, for $q>d$. It proves optimal regularity of the solutions at points $x_0\in\left\{Du\,=\,0\right\}$. A similar analysis in the context of nonvariational elliptic problem is pursued in \cite{heixeira}. Here, the author establishes Hessian continuity for the solutions to nonvariational elliptic equations at points $x_0\in\left\{D^2u\,=\,0\right\}$. See also \cite{edupota}.

We combine approximation methods with a localization argument and, inspired by the approach set forth in \cite{gradsing} and \cite{heixeira}, establish improved regularity of the solutions \emph{as they approach} $S_0(u)$. Our main result is the following: 
\begin{teo}\label{thm_main}
Let $u$ be a weak solution to \eqref{eq_main2}. Suppose A\ref{assump_source}, to be set forth in Section \ref{sec_mapm}, is in force. Given $\alpha\in(0,1)$ there exists $\varepsilon=\varepsilon(d,\alpha)$ such that, if $0<m-1<\varepsilon$, then $u\in\mathcal{C}^{\alpha,\frac{\alpha}{\sigma}}$ at $(x_0,t_0)\in S_0(u)\cap Q_{1/2}$, where $\sigma>0$ is a constant to be determined later. 

In addition, there exists $C>0$, depending only on $\alpha$ and the dimension $d$, for which
\[
	\sup_{B_r(x_0)\times(t_0-r^\sigma,t_0)} \left|u(x_0,t_0)\,-\,u(x,t)\right|\,\leq\,Cr^\alpha\left(\|u\|_{L^\infty(Q_1)}\,+\,\|f\|_{L^{p,q}(Q_1)}\right),
\]
for $0<r\ll 1/2$.
\end{teo}	

\begin{Remark}[Continuity of the regularity regime]
The constant $\sigma>0$ appearing in Theorem \ref{thm_main} is such that $\sigma\to 2$ as $m\to 1$. Hence, our result unveils a \emph{continuity} of the regularity regime for the PME equation with respect to the exponent $m$, in the sense it recovers the regularity available for the limiting profile, i.e., the heat equation. 
\end{Remark}

Although solutions to \eqref{eq_main2} are only $\mathcal{C}^\beta$-continuous in $Q_1$, Theorem \eqref{thm_main} ensures they \emph{land at $S_0(u)$ as almost Lipschitz-continuous functions}. Put differently, even if $u\in\mathcal{C}^\beta_{loc}(Q_1)$ with a very small exponent $0<\beta \ll 1/2$, the regularity regime switches to $\mathcal{C}^{1-}$ at the points where the density vanishes. A noticeable aspect of the theorem relies on the fact that improved regularity takes place along the set where no information from the PDE is available.

The proof of Theorem \ref{thm_main} is based on approximation methods, combined with a scaling argument. In fact, the approximation regime allows us to transmit information from the heat equation back to the PME. It translates into an oscillation control in cylinders of a universal, fixed, radius. Then, the scaling procedure localizes the oscillation estimate, establishing the result.  

The remainder of this paper is organized as follows. In Section \ref{sec_mapm} we detail our assumptions and collect some preliminaries. We prove a caloric approximation result in Section \ref{sec_caloric}. The proof of Theorem \ref{thm_main} is the subject of Section \ref{sec_proofmt}.

\section{A few preliminaries and the set up of the problem}\label{sec_mapm}


Throughout this paper we work under specific conditions on the source term $f$. We detail the anisotropic Lebesgue space to which the source term $f$ belongs.

\begin{Assumption}[Integrability of the source term $f$]\label{assump_source}
The source term $f : Q_1 \rightarrow \mathbb{R}$ is such that $f \in L^{p,q}(Q_1)$. In addition,
\[
	\left\|f\right\|_{L^{p,q}(Q_1)}\,:=\,\left(\int_{-1}^0\left|\int_{B_1}\left|f(x,t)\right|^pdx\right|^{q/p}dt\right)^{1/q} \leq C,
\]
where $C$ is a positive constant,
\begin{equation}\label{eq_intconditions}
	p\,>\,\frac{d}{2-m}\;\;\;\;\;\;\;\;\mbox{and}\;\;\;\;\;\;\;\;q\,\geq\,\frac{2(3\,-\,m)p}{(2\,-\,m)p\,-\,d}.
\end{equation}
\end{Assumption}

The conditions on the exponents $p$ and $q$ ensure that important quantities appearing further in the paper are non-negative. Namely,
\begin{equation}\label{eq_noneg1}
	\sigma\,:=\,2\,+\,(1\,-\,m)\alpha\,\geq\,0,
\end{equation}
for every $\alpha\in(0,1)$, and
\begin{equation}\label{eq_noneg2}
	\frac{\sigma m\,-\,2}{m\,-\,1}\,-\,\frac{d}{p}\,-\,\frac{\sigma}{q}\,\geq\,0.
\end{equation}
The quantity in \eqref{eq_noneg1} is related to the intrinsic scaling of the PME and encodes the H\"older regularity of the solutions with respect to $t$. On the other hand, the inequality in \eqref{eq_noneg2} arises in the context of a localization argument. Were $f\in L^\infty(Q_1)$, \eqref{eq_noneg2} could be disregarded. 

\begin{Remark}[Integrability conditions]
The lower bound for $q$ in \eqref{eq_intconditions} could be altered to
\begin{equation}\label{eq_condgamma}
	q\,\geq\,\frac{(1\,+\,\gamma)(3\,-\,m)p}{(2\,-\,m)p\,-\,d},
\end{equation}
for any $\gamma>0$. In fact, for any $\gamma>0$, condition \eqref{eq_condgamma} implies 
\[
	\frac{d}{p}\,+\,\frac{3\,-\,m}{q}\,<\,2\,-\,m.
\]
In the limit $m\to 1$, the former inequality yields
\[
	\frac{d}{p}\,+\,\frac{2}{q}\,<\,1,
\]
which is the condition found in \cite{jvteix} to ensure $\alpha$-H\"older continuity for solutions to parabolic fully nonlinear heat equation.
\end{Remark}


To make matters precise, we proceed with the definition of weak (distributional) solution to \eqref{eq_main2}.

\begin{Definition}[Weak solution]\label{def_solution}
A function $u\in L^\infty_{loc}(0,T;L^{m+1}_{loc}(B_1))$ with $u^m\in L^2_{loc}(0,T,W^{1,2}_{loc}(B_1))$ is said to be a weak (distributional) solution to \eqref{eq_main2} if
\[
	\int_0^T\int_{B_1}-u\phi_t\,+\,D(u^m)\cdot D\phi dxdt\,=\,\int_0^T\int_{B_1}f\phi dxdt,
\]
for every $\phi\in\mathcal{C}^\infty_0(Q_1)$. A normalized weak solution to \eqref{eq_main2} is a weak solution that satisfies
\[
	\left\|u\right\|_{L^\infty(Q_1)}\,\leq\,1.
\]
\end{Definition}

In \cite{urbdamma} the authors propose an alternative definition, which is tantamount to a third formulation involving the Steklov averages of $u$. The reason for this alternative approach lies in the use of a Caccioppoli estimate. We notice that our arguments by-pass the use of such inequality. Therefore, Definition \ref{def_solution} is reasonable in the context of the present work.


As mentioned before, we argue through techniques in the toolbox of geometrical tangential analysis. As mentioned before, the main elements behind this approach are stability of the solutions, compactness and scaling properties. In what follows we address the PME in light of those aspects. We begin with a proposition concerning the sequential stability of the solutions to \eqref{eq_main2}. 

\begin{Proposition}[Sequential stability]\label{prop_seqstab}
Consider the sequences $$(m_n)_{n\in\mathbb{N}}\subset(1,2),$$ $$(f_n)_{n\in\mathbb{N}}\subset L^{p,q}(Q_1)$$ and $$(u_n)_{n\in\mathbb{N}}\subset L^1_{loc}(Q_1).$$ Suppose
\[
	\left(u_n^{(m_n)}\right)_{n\in\mathbb{N}}\subset L^1_{loc}(0,T;W^{1,1}_{loc}(B_1)).
\]
Suppose further that $(u_n)_{n\in\mathbb{N}}$ solves
\[
(u_n)_t - \Delta(u_n^{m_n}) = f_n \;\;\;\text{in}\;\;\; Q_1.
\]
and
\[
	|m_n \,-\, 1| \,+ \,\|f_n\|_{L^{p,q}(Q_1)} \,\longrightarrow\, 0.
\]
If there exists $u_\infty \in L^1_{loc}(0,T;W^{1,1}(Q_1))\cap L^{\infty}(Q_1)$ such that
\[
	\left\|u_n\,-\,u_\infty\right\|_{L^\infty_{loc}(Q_{9/10})} \,\longrightarrow \,0,
\]
then $u_\infty$ solves
\[
	(u_\infty)_t\,-\,\Delta u_\infty\,=\,0\;\;\;\;\;\mbox{in}\;\;\;\;\;Q_{9/10}.
\]
\end{Proposition}
\begin{proof}
Let $\varphi \in \mathcal{C}^{\infty}_0(Q_1)$. We have
\begin{align*}
	\left|\int_{0}^T\int_{B_1}-u_\infty\varphi_t+Du_\infty\cdot D\varphi \,\mbox{d}x\mbox{d}t\right|&\leq\int_0^T\int_{B_1}\left|\varphi_t\right|\left|u_n-u_\infty\right|\,\mbox{d}x\mbox{d}t\\
		&\quad+\int_0^T\int_{B_1}\left|D^2\varphi\right|\left|u_n^{m_n}-u_\infty\right|\,\mbox{d}x\mbox{d}t\\
		&\quad+\int_0^T\int_{B_1}\left|f_n\varphi\right|\,\mbox{d}x\mbox{d}t.
\end{align*}
Notice that the right-hand side converges to zero as $n \to 0$. Therefore
\[
	\int_{Q_1}\left(u_\infty\varphi_t - Du_\infty D\varphi \right)\,\mbox{d}x\mbox{d}t = 0
\]
and the proof is complete.
\end{proof}

We continue by recalling a result on the compactness of the solutions. Indeed, the subsequent proposition accounts for the H\"older continuity of the solutions to \eqref{eq_main2} both in time and in space. It first appeared in the work of DiBenedetto and Friedman \cite{benfri}.

\begin{Proposition}[Compactness of the solutions]\label{prop_compac}
Let $u$ be a weak solution to \eqref{eq_main2}. Then $u\in\mathcal{C}^{\beta,\frac{\beta}{2}}_{loc}(Q_1)$, for some $\beta \in (0,1)$. In addition, there exists $C>0$, depending only on $\left\|u\right\|_{L^\infty(Q_1)}$ and $\left\|f\right\|_{L^{p,q}(Q_1)}$, such that
\[
	\left|u(x,t)\,-\,u(y,s)\right| \,\leq \,C\left(|x\,-\,y|^\beta\,+\,|t\,-\,s|^\frac{\beta}{2}\right).
\]
\end{Proposition}
For a proof of this proposition, we refer the reader to \cite{benfri}. We end this section by commenting on the scaling properties of the inhomogeneous porous medium equation. It is well known that if $u$ solves \eqref{eq_main2}, the function $v$ defined as
\begin{equation}\label{scal_func}
v(x,t) = \dfrac{u(x_0 + ax, t_0 + bt)}{\gamma},
\end{equation}
is itself a solution to \eqref{eq_main2} if and only if 
\begin{equation}\label{scal_prop}
\gamma = \left(\dfrac{a^2}{b}\right)^{\frac{1}{m-1}}.
\end{equation}
The proof of this fact is straightforward and we omit it here. By reasoning along the same lines, it is easy to see that if $u$ solves \eqref{eq_main2} and if $v$ and $\gamma$ are defined as in \eqref{scal_func} and \eqref{scal_prop}, then $v$ solves
\begin{equation}\label{scal_eq}
v_t - \Delta(v^m) = \dfrac{b^{\frac{m}{m-1}}}{a^{\frac{2}{m-1}}}f(x_0 + ax, t_0 + bt).
\end{equation}

Finally, we observe that imposing a smallness regime on the $L^{p,q}$-norm of $f$ or requiring $u$ to be a normalized solution accounts for no additional constraint on the problem. In fact, consider 
\[
	v\,:= \,\dfrac{u\left(\left(\frac{\|f\|_{L^{p,q}(Q_1)} + \|u\|_{L^\infty(Q_1)}}{\varepsilon}\right)^{\frac{m-1}{2}}x,t\right)}{\left(\frac{\|f\|_{L^{p,q}(Q_1)} + \|u\|_{L^\infty(Q_1)}}{\varepsilon}\right)}.
\]
Then $v$ is such that $\left\|v\right\|_{L^\infty(Q_1)}<1$ and satisfies
\[
	v_t\,-\,\Delta(v^m)\,=\,\frac{\varepsilon f}{\|f\|_{L^{p,q}(Q_1)}\,+\,\|u\|_{L^\infty(Q_1)}} \,=: \,{\tilde f},
\]
with $\|{\tilde f}\|_{L^{p,q}(Q_1)} < \varepsilon$. In the next section we present the proof of Theorem \ref{thm_main}.

\section{Caloric approximation for the solutions}\label{sec_caloric}

Improved regularity along the critical set $\left\lbrace u\,=\,0\right\rbrace$ relies on a finer approximation lemma. Indeed, this type of result ensures the existence of an auxiliary function $h$ approximating the solutions, with higher levels of regularity. 

Heuristically, we aim at designing a Taylor expansion for the solutions, importing regularity from $h$. In our concrete case, it is critical that $h$ vanishes at the points $(x,t)\in S_0(u)$. That is, we must guarantee 
\[
	\left\lbrace(x,t)\in Q_1\;|\; u(x,t)\,=\,0\right\rbrace\,\subset\,\left\lbrace(x,t)\in Q_1\;|\; h(x,t)\,=\,0\right\rbrace,
\]
where $h$ is the approximating function.

\begin{Proposition}[Zero level-set approximation lemma]\label{prop_approx}
Let $u$ be a weak solution to \eqref{eq_main2}. Suppose A\ref{assump_source} is in force. Suppose further that $(x_0,t_0) \in S_0(u)$. Then, given $\delta > 0$, there exists $\varepsilon>0$ such that if 
\[
	m\,-\,1 \,+ \,\|f\|_{L^{p,q}(Q_1)} \,<\,\varepsilon,
\]
one can find $h \in \mathcal{C}^{1,\frac{1}{2}}(Q_1)$ satisfying
\[
	\left\|u\,-\,h\right\|_{L^{\infty}(Q_{9/10})}\,\leq\,\delta,
\]
with $h(x_0,t_0) \,= \,0$
\end{Proposition}
\begin{proof}
We argue by contradiction. Suppose the statement of the proposition is false. In this case, there exist $\delta_0$ and sequences $(u_n)_{n\in\mathbb{N}}$, $(m_n)_{n\in\mathbb{N}}\subset (1,2)$ and $(f_n)_{n\in\mathbb{N}}$, such that
\[
	(m_n-1)\, +\, \left\|f_n\right\|_{L^{p,q}(Q_1)}\,\leq\,\frac{1}{n}
\]
and
\begin{equation}\label{eq_seq}
	(u_n)_t\,-\,\Delta(u^m_n) \,= \,f_n\;\;\;\;\;\mbox{in}\;\;\;\;\;Q_1,
\end{equation}
with
\[
\|u_n-h\|_{L^{\infty}(Q_1)} > \delta_0\;\;\;\;\;\;\;\;\;\;\mbox{or}\;\;\;\;\;\;\;\;\;\;h(x_0,t_0)\,\neq\,0
\]
for all $h\in C^{2,1}(Q_1)$ and $n \in \mathbb{N}$. Proposition \ref{prop_compac} implies that $(u_n)_{n\in\mathbb{N}}$ is uniformly bounded in $C^{\beta,\frac{\beta}{2}}(Q_1)$, for some $\beta \in (0,1)$. Therefore, $u_n \rightarrow u_\infty$ locally in $C^{\gamma,\frac{\gamma}{2}}(Q_1)$ through a subsequence, if necessary, for every $0<\gamma<\beta$. Here we evoke the sequential stability of weak solutions, Proposition \ref{prop_seqstab}, to conclude that $u_\infty$ solves
\[
	(u_\infty)_t\,-\,\Delta(u_{\infty}) \,= \,0 \;\;\;\;\;\mbox{in}\;\;\;\;\;Q_{9/10}.
\]
Standard regularity results available for the heat equation ensure that $u_\infty \in \mathcal{C}^{1,\frac{1}{2}}(Q_1)$. In addition, as a consequence of the uniform convergence, $u_\infty(x_0,t_0)=0$. Now, by taking $h\equiv u_\infty$, we obtain a contradiction and complete the proof.
\end{proof}

In the sequel we detail the proof of Theorem \ref{thm_main}.

\section{Improved regularity along singular sets}\label{sec_proofmt}

We start by controlling the oscillation of the solutions to \eqref{eq_main2} within a ball of radius $0<\rho\ll1/2$, to be (universally) determined.

\begin{Proposition}\label{prop_step1}
Let $u$ be a weak solution to \eqref{eq_main2}. Suppose that A\ref{assump_source} is in force. Suppose further that  $(x_0,t_0) \in S_0(u)$. Then, given $\alpha \in (0,1)$, there exists $\varepsilon>0$ such that, if
\[
	(m\,-\,1)\,+\,\left\|f\right\|_{L^{p,q}(Q_1)}\,<\,\varepsilon
\]
one can find a constant $0<\rho\ll1/2$ for which
\[
	\sup_{Q_{\rho}(x_0,t_0)} |u(x,t)|\,\leq \,\rho^\alpha.
\]
\end{Proposition}
\begin{proof}
By Proposition \ref{prop_approx}, there exists $h \in C^{2,1}(Q_1)$ such that for all $\delta >0$
\[
\|u-h\|_{L^{\infty}(Q_1)} \leq \delta,
\]
with $h(x_0,t_0)$ = 0. We get
\begin{equation}\label{eq_hconst}
\sup_{Q_{\rho}(x_0,t_0)}|h(x,t) - h(x_0,t_0)| \leq C\rho.
\end{equation}
Thus,
\[
\begin{array}{rcl}
\displaystyle \sup_{Q_{\rho}(x_0,t_0)}|u(x,t)| & = & \displaystyle \sup_{Q_{\rho}(x_0,t_0)}|u(x,t) - h(x,t) + h(x,t)| \vspace{0.1cm}\\
 & \leq & \displaystyle \sup_{Q_{\rho}(x_0,t_0)}|u(x,t) - h(x,t)| \vspace{0,1cm}\\
 & & \displaystyle + \sup_{Q_{\rho}(x_0,t_0)}|h(x,t) - h(x_0,t_0)| \vspace{0,1cm} \\
 & \leq & \delta + C\rho.
\end{array}
\]
Now, we make the following (universal) choices
\[
\rho := \left(\dfrac{1}{2C}\right)^{\frac{1}{1-\alpha}} \;\;\;\text{and}\;\;\; \delta:=\frac{\rho^\alpha}{2}.
\]
Such choices lead to
\[
\sup_{Q_{\rho}(x_0,t_0)}|u(x,t)| \leq \rho^{\alpha},
\]
which finishes the proof.
\end{proof}

\begin{Remark}
\normalfont Note that the constant $C>0$ in \eqref{eq_hconst} depends only on the dimension and the $L^\infty$-norm of $u$ in $Q_1$. In fact, that $h$ solves a heat equation with initial-boundary data given by $u$. Therefore, given $\alpha\in(0,1)$, arbitrarily, the choice of $\rho$ depends only on $\alpha$. For that reason, the proximity regime $\delta$ depends solely on the exponent $\alpha$. Given $\delta$, we find $\varepsilon$, being also dependent on $\alpha$ alone. These universal choices set the smallness regime in Proposition \ref{prop_approx}.
\end{Remark}

In the sequel we refine Proposition \ref{prop_step1}. This is done by producing an oscillation control at discrete scales of the form $(\rho^n)_{n\in\mathbb{N}}$. Before we proceed, we introduce some notation. The scaled parabolic cylinder $\overline{Q}_\rho(x_0,t_0)$ is given by
\[
	\overline{Q}_\rho(x_0,t_0)\,:=\,B_{\rho}(x_0)\,\times\,\left(t_0-\rho^{\sigma},t_0+\rho^{\sigma}\right).
\]
\begin{Proposition}\label{prop_int}
Let $u$ be a normalized weak solution to \eqref{eq_main2}. Suppose A\ref{assump_source} is in force. Suppose further that $(x_0,t_0) \in S_0(u)$. Then, for every $\alpha \in (0,1)$, there exists $\varepsilon>0$ such that, if
\[
	(m\,-\,1)\,+\,\left\|f\right\|_{L^{p,q}(Q_1)}\,<\,\varepsilon,
\]
then
\[
	\sup_{\overline{Q}_{\rho^n}(x_0,t_0)}\left|u(x,t)\right|\,\leq\,\rho^{n\alpha},
\]
for every $n \in \mathbb{N}$
\end{Proposition}
\begin{proof}
We prove the proposition by induction. Notice that Proposition \ref{prop_step1} accounts for the first step in the induction argument. Suppose we have verified the statement for $n = k$. It remains to verify it in the case $n = k + 1$. Consider the function
\[
v(x,t) := \dfrac{u(x_0 + \rho^k x, t_0 + \rho^{k\sigma}t)}{\rho^{k\alpha}},
\]
where 
\[
\sigma := 2+(1-m)\alpha.
\]
Because of our assumptions, we have $\sigma >0$. Furthermore, the induction hypothesis guarantees that $\|v\|_{L^{\infty}(Q_1)} \leq 1$. The scaling properties detailed in \eqref{scal_func}, \eqref{scal_prop} and \eqref{scal_eq} imply that $v$ solves 
\[
	v_t\,-\,\Delta(v^m)\,=\,\frac{\rho^{\frac{k\sigma m}{m-1}}}{\rho^{\frac{2k}{m-1}}}f \,=: \,{\tilde f}.
\]
In addition,
\[
	\left\|{\tilde f}\right\|_{L^{p,q}(Q_1)} = \rho^{k\left(\frac{\sigma m - 2}{m-1}-\frac{d}{p}-\frac{\sigma}{q}\right)}\left\|f\right\|_{L^{p,q}(Q_{\rho^k})}\,<\,\varepsilon,
\]
since the conditions in A\ref{assump_source} ensure \eqref{eq_noneg2} is available.

We conclude that $v$ falls under Proposition \ref{prop_step1}. Hence, we obtain
\[
	\sup_{Q_{\rho}(0,0)}\left|v(x,t)\right| \,\leq \,\rho^\alpha,
\]
since $(0,0)\in S_0[v]$. This implies that
\begin{align*}
	\sup_{\overline{Q}_\rho(x_0,t_0)}|u(x,t)| \,& \leq \,\rho^{\alpha}\rho^{\frac{2k}{m-1}} \,\leq\,\rho^{\alpha}\rho^{\alpha k} \,\leq\,\rho^{(k+1)\alpha},
\end{align*}
where 
\[
	\overline{Q}_\rho(x_0,t_0)\,:=\,B_{\rho^{k+1}}(x_0)\,\times\,\left(t_0-\rho^{(k+1)\sigma},t_0+\rho^{(k+1)\sigma}\right),
\]
and the proof is complete.
\end{proof}

We proceed with the proof of Theorem \ref{thm_main}. In light of Proposition \ref{prop_int}, it remains to produce a discrete-to-continuous argument, extending the oscillation control to the radii $0<r\ll 1/2$. 

\begin{proof}[Proof of Theorem \ref{thm_main}]
Let $0<r\ll 1/2$ be fixed, though arbitrary. Suppose $(x_0,t_0) \in S_0(u)$ and fix $n \in \mathbb{N}$ such that $\rho^{n+1} \leq r \leq \rho^n$. From Proposition \ref{prop_step1} we infer that 
\[
	\sup_{\overline{Q}_{\rho^n}(x_0,t_0)}\left|u(x,t)\,-\,u(x_0,t_0)\right|\,\leq\,\rho^{n\alpha}\,=\,\rho^{(n+1)\alpha}\rho^{-\alpha}.
\]
Therefore,
\begin{align*}
	\sup_{\overline{Q}_r(x_0,t_0)}\left|u(x,t)\,-\,u(x_0,t_0)\right|\,&\leq\,\sup_{\overline{Q}_{\rho^n}(x_0,t_0)}\left|u(x,t)\,-\,u(x_0,t_0)\right|\\
		&\leq\,\rho^{(n+1)\alpha}\rho^{-\alpha}\\
		&\leq\, Cr^\alpha,
\end{align*}
which ends the argument.
\end{proof}

\bibliography{biblio}
\bibliographystyle{plain}

\bigskip

\noindent\textsc{Edgard A. Pimentel (Corresponding Author)}\\
Department of Mathematics\\
Pontifical Catholic University of Rio de Janeiro -- PUC-Rio\\
22451-900, G\'avea, Rio de Janeiro-RJ, Brazil\\
\noindent\texttt{pimentel@puc-rio.br}

\bigskip

\noindent\textsc{Makson S. Santos}\\
Department of Mathematics\\
Pontifical Catholic University of Rio de Janeiro -- PUC-Rio\\
22451-900, G\'avea, Rio de Janeiro-RJ, Brazil\\
\noindent\texttt{makson.santos@mat.puc-rio.br}

\end{document}